\documentclass[pdflatex,sn-mathphys-num]{sn-jnl}

\usepackage{amsmath,amssymb,amsfonts}
\usepackage{mathtools}
\usepackage[nameinlink,noabbrev]{cleveref}

\crefname{theorem}{Theorem}{Theorems}
\crefname{proposition}{Proposition}{Propositions}
\crefname{lemma}{Lemma}{Lemmas}
\crefname{remark}{Remark}{Remarks}
\crefname{section}{Section}{Sections}

\numberwithin{equation}{section}

\theoremstyle{thmstyleone}
\newtheorem{theorem}{Theorem}[section]
\newtheorem{proposition}[theorem]{Proposition}
\newtheorem{lemma}[theorem]{Lemma}

\theoremstyle{thmstyletwo}
\newtheorem{remark}[theorem]{Remark}

\newcommand{\Pf}{\operatorname{Pf}}
\newcommand{\diag}{\operatorname{diag}}
\newcommand{\one}{\mathbf{1}}
\newcommand{\T}{{\mathsf T}}
\newcommand{\ee}{\mathrm{e}}
\newcommand{\ii}{\mathrm{i}}
\newcommand{\floor}[1]{\left\lfloor #1\right\rfloor}

\begin{document}

\title[A truncated Legendre-symbol determinant]{A Proof of a Conjecture of Zhi-Wei Sun on a Truncated Legendre-Symbol Determinant}

\author[1]{\fnm{Yaoran} \sur{Yang}}\email{yangyaoran@stu.scu.edu.cn}
\author[2]{\fnm{Gaishi} \sur{Yang}}\email{dream@mails.ccnu.edu.cn}
\author*[1]{\fnm{Yutong} \sur{Zhang}}\email{yutongzhang@stu.scu.edu.cn}

\affil[1]{\orgdiv{School of Mathematics}, \orgname{Sichuan University},
\orgaddress{\city{Chengdu}, \state{Sichuan}, \country{China}}}

\affil[2]{\orgdiv{School of Mathematics and Statistics},
\orgname{Central China Normal University},
\orgaddress{\city{Wuhan}, \state{Hubei}, \country{China}}}

\abstract{Let $p\ge7$ be a prime with $p\equiv3\pmod4$ and let $\chi$ be the
Legendre symbol modulo $p$.  We prove that
$\det[x+\chi(j-k)]_{0\le j,k\le(p-7)/2}=\floor{(p-2)/3}^{2}x$ in
$\mathbb{C}[x]$, which settles Conjecture~3.4 of Zhi-Wei Sun.  The truncated matrix is a corner of Chapman's Legendre-symbol
matrix $C$, and its determinant can be expressed through a handful of entries
of $C^{-1}$ and of $C^{-1}\one$.  Those entries are in turn read off from
Vsemirnov's cyclotomic factorization of $C$ with the help of Schur's Pfaffian
identity.}



\maketitle

\section{Introduction}

For an odd prime $p$ let $\chi(a)=\left(\frac ap\right)$ be the Legendre
symbol, with the convention $\chi(0)=0$.  Determinants formed from the values
$\chi(j-i)$ have been studied since Chapman~\cite{Chapman2004}, who evaluated
several families of them.  The one that has attracted the most attention is
the ``evil determinant'' of
\[
 C=\bigl[\chi(j-i)\bigr]_{0\le i,j\le (p-1)/2}.
\]
For $p\equiv 3\pmod 4$, Vsemirnov~\cite{Vsemirnov2012} proved that $\det C=1$
and, what matters more here, obtained an explicit factorization of $C$ into
cyclotomic data; we use the normalization recorded in
\cite[Remark~1, Eq.~(2.6)]{Vsemirnov2013}.

Sun~\cite{Sun2024} has assembled a long list of problems and conjectures on
determinants involving Legendre symbols.  The present note settles
\cite[Conjecture~3.4]{Sun2024}.

\begin{theorem}\label{thm:main}
Let $p\ge 7$ be a prime with $p\equiv 3\pmod 4$.  Then, as an identity in
$\mathbb{C}[x]$,
\begin{equation}\label{eq:main}
 \det\!\left[x+\chi(j-k)\right]_{0\le j,k\le (p-7)/2}
 =\floor{\frac{p-2}{3}}^{\!2}x.
\end{equation}
\end{theorem}

Throughout we fix such a $p$ and write
\[
 p=2n+1,\qquad m=\floor{p/6}.
\]
Then $n=(p-1)/2$ is odd, so that $\chi(-1)=-1$ and $C$ is skew-symmetric of
even order $n+1$.  Since $p>3$ is prime we have $p\equiv\pm1\pmod 6$: either
$p=6m+1$ and $n=3m$, or $p=6m+5$ and $n=3m+2$.  Checking the two cases
separately,
\begin{equation}\label{eq:floor}
 \floor{\frac{p-2}{3}}=n-1-m,
\end{equation}
and it is in this shape that the right-hand side of \cref{eq:main} will
appear.

Apart from a sign and the rank-one term $x\one\one^{\T}$, where $\one$ denotes
an all-ones column vector, the matrix in \cref{eq:main} is a principal
submatrix of $C$ of odd order $n-2$.  Jacobi's identity for the minors of an
inverse, used twice and combined with the Sherman--Morrison formula, therefore
collapses the determinant into the square of a single coordinate of
$C^{-1}\one$, and that coordinate is a three-term combination of entries of
$C^{-1}$ and $C^{-1}\one$; this is \cref{sec:reduction}.  Vsemirnov's
factorization then replaces $C^{-1}$ by the inverse of a matrix $K$ of the
shape occurring in Schur's Pfaffian identity (\cref{sec:vsemirnov}), and
$K^{-1}$ is reached by adjoining two extra nodes to that identity
(\cref{sec:kernel}).  All that is left is to read the required entries off the
first Taylor coefficients of a single rational function
(\cref{sec:evaluation}).

\section{Reduction to the inverse of Chapman's matrix}\label{sec:reduction}

Write $W=C^{-1}$, which is skew-symmetric because $C$ is.  Vsemirnov's theorem
for $p\equiv3\pmod4$ gives
\begin{equation}\label{eq:detC}
 \det C=1.
\end{equation}
The three quantities that carry the whole argument are
\[
 d=W_{0,n},\qquad r=W_{0,n-1},\qquad \gamma=W\one .
\]

\begin{proposition}\label[proposition]{prop:inverse}
With the notation above,
\begin{equation}\label{eq:inverse-data}
 d^{2}=\gamma_n^{2}=1,\qquad
 r=-md,\qquad
 \gamma_0=-\gamma_n,\qquad
 \gamma_{n-1}=-n\gamma_n .
\end{equation}
\end{proposition}

The proof of \cref{prop:inverse} is cyclotomic and occupies
\cref{sec:vsemirnov,sec:kernel,sec:evaluation}.  Here we deduce the theorem
from it.

\begin{proposition}\label[proposition]{prop:reduction}
\Cref{prop:inverse} implies \cref{thm:main}.
\end{proposition}

\begin{proof}
Let $H$ and $T$ be the principal submatrices of $C$ on the index sets
$\{0,1,\ldots,n-2\}$ and $\{0,1,\ldots,n-3\}$.  Thus $H$ has even order $n-1$
and $T$ has odd order $n-2$.

\emph{Step 1: $W_{n-1,n}=d$.}  For an invertible skew-symmetric matrix $A$
indexed from $0$ and for $i<j$, the Pfaffian cofactor formula reads
\[
 (A^{-1})_{ij}=\pm(-1)^{i+j}\frac{\Pf A_{\widehat{i,j}}}{\Pf A},
\]
where $A_{\widehat{i,j}}$ is obtained by deleting rows and columns $i$ and
$j$, and the overall sign depends on the convention chosen for the Pfaffian.
Only the ratio of two such expressions will occur, so that convention is
immaterial.  Deleting rows and columns $0$ and $n$ from $C$ leaves the ordered
index set $\{1,\ldots,n-1\}$, whereas deleting $n-1$ and $n$ leaves
$\{0,\ldots,n-2\}$; since $C_{ij}$ depends only on $j-i$, the two submatrices
are literally equal.  Their Pfaffians therefore agree, and so do the two
signs, because $n$ is odd and hence $(-1)^{0+n}=(-1)^{(n-1)+n}=-1$.
Consequently
\begin{equation}\label{eq:shift}
 W_{n-1,n}=W_{0,n}=d.
\end{equation}

\emph{Step 2: $\det H=1$.}  Jacobi's identity for the minors of an inverse
states that $\det A[I,I]=\det A\cdot\det A^{-1}[I^{c},I^{c}]$ for invertible
$A$.  Taking $A=C$ and $I=\{0,\ldots,n-2\}$ and using \cref{eq:detC,eq:shift},
\[
 \det H=\det C\cdot\det\begin{pmatrix}0&d\\-d&0\end{pmatrix}=d^{2}=1 .
\]
In particular $H$ is invertible, and we may set $u=H^{-1}\one$.

\emph{Step 3: $\det(T+\one\one^{\T})=u_0^{2}$.}  The inverse $H^{-1}$ is again
skew-symmetric, so $\one^{\T}H^{-1}\one=0$ and $\one^{\T}H^{-1}=-u^{\T}$.
Writing $B=H+\one\one^{\T}$, the matrix determinant lemma and the
Sherman--Morrison formula give
\[
 \det B=\det H\bigl(1+\one^{\T}H^{-1}\one\bigr)=1,
 \qquad
 B^{-1}=H^{-1}-\frac{H^{-1}\one\one^{\T}H^{-1}}
 {1+\one^{\T}H^{-1}\one}=H^{-1}+uu^{\T}.
\]
Deleting the last row and column of $B$ leaves $T+\one\one^{\T}$, so the
one-index case of Jacobi's identity yields
\[
 \det(T+\one\one^{\T})=\det B\cdot(B^{-1})_{n-2,n-2}=u_{n-2}^{2},
\]
the diagonal entry of $H^{-1}$ being zero.  Let now $\mathcal R$ be the
reversal permutation matrix of order $n-1$.  As $H$ is Toeplitz and
$\chi(-1)=-1$, we have $\mathcal RH\mathcal R=-H$, that is,
$H\mathcal R=-\mathcal RH$; since $\mathcal R\one=\one$, this gives
$H(\mathcal Ru)=-\mathcal RHu=-\one$, whence $\mathcal Ru=-u$ and
$u_{n-2}=-u_0$.

\emph{Step 4: $u_0=\gamma_0-\gamma_{n-1}+(r/d)\gamma_n$.}  Extend $u$ by two
zeros to $\hat u=(u_0,\ldots,u_{n-2},0,0)^{\T}\in\mathbb C^{n+1}$.  Only the
block $H$ acts on the nonzero coordinates of $\hat u$, so the first $n-1$
coordinates of $C\hat u$ are those of $Hu=\one$.  Hence $C\hat u=\one-w$ for
some vector $w$ supported on $\{n-1,n\}$, and applying $W$ gives
$\hat u=\gamma-Ww$.  Coordinates $n-1$ and $n$ of $\hat u$ vanish, and by
\cref{eq:shift} together with the skew-symmetry of $W$ they read
\[
 0=\gamma_{n-1}-dw_n,\qquad 0=\gamma_n+dw_{n-1}.
\]
Since $d^{2}=1$ we have $d\ne0$, so $w_n=\gamma_{n-1}/d$ and
$w_{n-1}=-\gamma_n/d$.  Coordinate $0$ of $\hat u=\gamma-Ww$ now gives
\[
 u_0=\gamma_0-\bigl(rw_{n-1}+dw_n\bigr)
 =\gamma_0-\gamma_{n-1}+\frac rd\,\gamma_n .
\]

\emph{Step 5: conclusion.}  Feeding \cref{eq:inverse-data} into Step 4 yields
$u_0=(n-1-m)\gamma_n$, so that, by \cref{eq:floor} and $\gamma_n^{2}=1$,
\begin{equation}\label{eq:TJ}
 \det(T+\one\one^{\T})=u_0^{2}=(n-1-m)^{2}\gamma_n^{2}
 =\floor{\frac{p-2}{3}}^{\!2}.
\end{equation}
Expanding by multilinearity in the columns,
$\det(T+y\one\one^{\T})=\det T+y\,\one^{\T}\!\operatorname{adj}(T)\one$, and
$\det T=0$ because $T$ is skew-symmetric of odd order; hence
$\det(T+y\one\one^{\T})=y\det(T+\one\one^{\T})$ for every $y$.  Finally
$T_{jk}=\chi(k-j)=-\chi(j-k)$, so the matrix in \cref{eq:main} is
$x\one\one^{\T}-T$, of odd order $n-2$, and
\[
 \det(x\one\one^{\T}-T)=-\det(T-x\one\one^{\T})=x\det(T+\one\one^{\T}).
\]
Together with \cref{eq:TJ} this is precisely \cref{eq:main}.
\end{proof}

\section{Vsemirnov's factorization and Schur's Pfaffian
identity}\label{sec:vsemirnov}

Let $\zeta=\ee^{2\pi\ii/p}$ and, for $a\in\mathbb Z$,
\[
 \tau(a)=\sum_{t=1}^{p-1}\chi(t)\zeta^{at},\qquad \tau=\tau(1),
\]
so that $\tau(a)=\chi(a)\tau$ whenever $p\nmid a$, and
\begin{equation}\label{eq:gauss}
 \tau^{2}=\chi(-1)p=-p;
\end{equation}
see for instance \cite[Chapter~6]{IrelandRosen}.  Put $x_i=\zeta^{2i}$ for
$0\le i\le n$.  These are pairwise distinct, since $2i\equiv2j\pmod p$ forces
$i=j$ in the range considered.  With $g(X)=\prod_{i=0}^{n}(X-x_i)$, define
matrices $V,D,U$ of order $n+1$ by
\begin{gather*}
 V_{ij}=\zeta^{2ij}=x_i^{\,j},
 \qquad
 D=\diag\!\left(\frac1{g'(x_0)},\ldots,\frac1{g'(x_n)}\right),\\
 U_{ij}=\frac{\chi(i)\zeta^{-j-2i}-\chi(j)\zeta^{-2j-i}}
 {\zeta^{-i-j}-\chi(i)\chi(j)} .
\end{gather*}
Thus $V$ is the Vandermonde matrix of the nodes $x_i$, symmetric and
invertible, while $U$ is skew-symmetric.  Vsemirnov's theorem for
$p\equiv3\pmod4$ asserts \cref{eq:detC} together with the factorization
\begin{equation}\label{eq:vsemirnov}
 C=\lambda VDUDV,\qquad \lambda=-\tau(2)\zeta^{-(p+1)/4},
\end{equation}
the exponent being an integer because $p\equiv3\pmod4$; see
\cite{Vsemirnov2012} and \cite[Remark~1, Eq.~(2.6)]{Vsemirnov2013}.

A diagonal rescaling turns $U$ into a matrix of Cauchy type.  Set
\[
 a_0=1,\qquad a_i=\chi(i)\zeta^{i}\quad(1\le i\le n),
 \qquad G=\diag(a_0,a_1,\ldots,a_n),
 \qquad K=GUG.
\]

\begin{lemma}\label[lemma]{lem:K}
The numbers $a_0,\ldots,a_n$ are pairwise distinct, and $a_ia_j=1$ only for
$i=j=0$.  All denominators in the definition of $U$ are nonzero, and
\[
 K_{ij}=\frac{a_j-a_i}{1-a_ia_j}\quad(i\ne j),\qquad K_{ii}=0 .
\]
\end{lemma}

\begin{proof}
Squaring gives $a_i^{2}=\zeta^{2i}=x_i$, valid for $i=0$ as well, so $a_i=a_j$
forces $x_i=x_j$ and hence $i=j$.  Likewise $a_ia_j=1$ implies
$\zeta^{2(i+j)}=1$, so $p\mid i+j$; as $0\le i+j\le2n=p-1$, only $i=j=0$
survives.

For the denominators, note that $\zeta^{-i-j}-\chi(i)\chi(j)
=\zeta^{-i-j}(1-a_ia_j)$ when $i,j\ge1$, which is then nonzero, while if $i=0$
or $j=0$ the denominator reduces to a power of $\zeta$ because $\chi(0)=0$.
The same computation delivers the entries: for $i,j\ge1$,
\[
 K_{ij}=a_ia_jU_{ij}
 =\frac{\chi(j)\zeta^{-i}-\chi(i)\zeta^{-j}}
 {\zeta^{-i-j}-\chi(i)\chi(j)},
\]
and multiplying numerator and denominator by $\zeta^{i+j}$ produces
$(a_j-a_i)/(1-a_ia_j)$.  For $i=0<j$ one finds $U_{0j}=-\chi(j)\zeta^{-j}$ and
hence $K_{0j}=-1$, which is again $(a_j-a_0)/(1-a_0a_j)$.  Finally $U_{ii}=0$
for every $i$, so $K$ has zero diagonal.
\end{proof}

Schur's Pfaffian identity~\cite{Schur1911} states that
\[
 \Pf\!\left[\frac{y_j-y_i}{y_j+y_i}\right]_{1\le i,j\le2M}
 =\prod_{1\le i<j\le 2M}\frac{y_j-y_i}{y_j+y_i}.
\]
Under the substitution $b_i=(y_i-1)/(y_i+1)$ one computes
\[
 b_j-b_i=\frac{2(y_j-y_i)}{(y_i+1)(y_j+1)},
 \qquad
 1-b_ib_j=\frac{2(y_i+y_j)}{(y_i+1)(y_j+1)},
\]
so that the quotient of the two is $(y_j-y_i)/(y_i+y_j)$.  Being a birational
change of variables, the substitution converts Schur's identity into
\begin{equation}\label{eq:schur}
 \Pf\!\left[\frac{b_j-b_i}{1-b_ib_j}\right]_{1\le i,j\le2M}
 =\prod_{1\le i<j\le 2M}\frac{b_j-b_i}{1-b_ib_j},
\end{equation}
an identity of rational functions in $b_1,\ldots,b_{2M}$, the diagonal entries
on the left being $0$.

The order $n+1$ of $K$ is even, so \cref{eq:schur} applies to the nodes
$a_0,\ldots,a_n$; by \cref{lem:K} every factor of the resulting product is
finite and nonzero, so
\begin{equation}\label{eq:pfK}
 \Pf K=\prod_{0\le i<j\le n}\frac{a_j-a_i}{1-a_ia_j}\ne0 .
\end{equation}
In particular $K$ is invertible.  Substituting $U=G^{-1}KG^{-1}$ into
\cref{eq:vsemirnov} and inverting,
\begin{equation}\label{eq:Winv}
 W=\lambda^{-1}V^{-1}D^{-1}GK^{-1}GD^{-1}V^{-1}.
\end{equation}

\section{A two-point identity for $K^{-1}$}\label{sec:kernel}

For a variable $z$ set
\[
 k(z)=\left(\frac{z-a_i}{1-a_iz}\right)_{i=0}^{n},
 \qquad
 P(z)=\prod_{i=0}^{n}\frac{z-a_i}{1-a_iz}.
\]

\begin{lemma}\label[lemma]{lem:kernel}
As an identity of rational functions in $z$ and $w$,
\begin{equation}\label{eq:kernel}
 k(z)^{\T}K^{-1}k(w)=\frac{w-z}{1-zw}\bigl(P(z)P(w)-1\bigr).
\end{equation}
\end{lemma}

\begin{proof}
Adjoin $z$ and $w$ to the nodes $a_0,\ldots,a_n$, in that order.  The
associated Schur matrix
\[
 \mathcal K=
 \begin{pmatrix}
 K&k(z)&k(w)\\
 -k(z)^{\T}&0&\dfrac{w-z}{1-zw}\\[2pt]
 -k(w)^{\T}&-\dfrac{w-z}{1-zw}&0
 \end{pmatrix}
\]
has even order $n+3$, so \cref{eq:schur,eq:pfK} give
\[
 \frac{\Pf\mathcal K}{\Pf K}=\frac{w-z}{1-zw}P(z)P(w).
\]
On the other hand, an invertible skew-symmetric $A$ admits a Pfaffian
Schur-complement formula,
\begin{equation}\label{eq:pfschur}
 \Pf\begin{pmatrix}A&B\\-B^{\T}&D_0\end{pmatrix}
 =\Pf(A)\,\Pf\bigl(D_0+B^{\T}A^{-1}B\bigr).
\end{equation}
Indeed, with $S=\begin{pmatrix}I&0\\B^{\T}A^{-1}&I\end{pmatrix}$ one checks,
using $(A^{-1})^{\T}=-A^{-1}$, that
\[
 S\begin{pmatrix}A&B\\-B^{\T}&D_0\end{pmatrix}S^{\T}
 =\begin{pmatrix}A&0\\0&D_0+B^{\T}A^{-1}B\end{pmatrix},
\]
and it remains to apply $\Pf(SMS^{\T})=\det(S)\Pf(M)$ with $\det S=1$.  Take
$A=K$ and $B=(k(z),k(w))$ in \cref{eq:pfschur}.  The diagonal of
$B^{\T}K^{-1}B$ vanishes because $K^{-1}$ is skew-symmetric, so
\[
 \frac{\Pf\mathcal K}{\Pf K}=\frac{w-z}{1-zw}+k(z)^{\T}K^{-1}k(w),
\]
and comparing the two evaluations proves \cref{eq:kernel}.
\end{proof}

The zeroth entry of $k(z)$ is the constant $-1$, and $n+1$ is even; hence,
with $e_0$ the first standard basis vector,
\[
 k(-1)=-\one,\quad P(-1)=1,
 \qquad
 k(1)=\one-2e_0,\quad P(1)=-1 .
\]
Substituting $w=-1$ and then $w=1$ in \cref{eq:kernel} and combining the two
results through $e_0=-\tfrac12\bigl(k(-1)+k(1)\bigr)$, we obtain
\begin{equation}\label{eq:e0}
 k(z)^{\T}K^{-1}e_0=P(z).
\end{equation}
Write
\[
 \Delta=\prod_{i=0}^{n}a_i,
 \qquad
 v_\mu=(a_i^{\mu})_{i=0}^{n}\quad(\mu\in\mathbb Z).
\]
Again because $n+1$ is even, $k(\infty)=-v_{-1}$ and $P(\infty)=\Delta^{-1}$,
so evaluating \cref{eq:kernel} at $z=\infty$ gives
\begin{equation}\label{eq:infty}
 v_{-1}^{\T}K^{-1}k(w)=\frac{1-\Delta^{-1}P(w)}{w}.
\end{equation}

We shall need the first Taylor coefficients of $P$ at the origin.  Since $n+1$
is even,
\[
 P(z)=\Delta R(z),
 \qquad
 R(z)=\prod_{i=0}^{n}\frac{1-a_i^{-1}z}{1-a_iz},
\]
and with $\delta_\mu=\sum_{i=0}^{n}(a_i^{\mu}-a_i^{-\mu})$ one has
$\log R(z)=\sum_{\mu\ge1}\delta_\mu z^{\mu}/\mu$ as a formal power series.
Exponentiating and writing $R(z)=1+\rho_1z+\rho_2z^{2}+\rho_3z^{3}+O(z^{4})$,
\begin{equation}\label{eq:rho}
 \rho_1=\delta_1,
 \qquad
 \rho_2=\frac{\delta_1^{2}+\delta_2}{2},
 \qquad
 \rho_3=\frac{\delta_1^{3}+3\delta_1\delta_2+2\delta_3}{6}.
\end{equation}

The three power sums involved are classical.  The term $i=0$ contributes
nothing to $\delta_\mu$, and pairing the residues $i$ and $-i$ modulo $p$
identifies what is left as Gauss sums:
\begin{equation}\label{eq:delta13}
 \delta_1=\sum_{i=1}^{n}\chi(i)(\zeta^{i}-\zeta^{-i})=\tau,
 \qquad
 \delta_3=\sum_{i=1}^{n}\chi(i)(\zeta^{3i}-\zeta^{-3i})=\tau(3)=\chi(3)\tau .
\end{equation}
For $\delta_2$, two geometric series suffice.  Putting
$S=\sum_{i=0}^{n}a_i^{2}=\sum_{i=0}^{n}\zeta^{2i}$, we get
$S=(1-\zeta^{p+1})/(1-\zeta^{2})=1/(1+\zeta)$ and, in the same way,
$\sum_{i=0}^{n}a_i^{-2}=\zeta/(1+\zeta)$, so that
$\delta_2=(1-\zeta)/(1+\zeta)$.  Only one combination of $S$ and $\delta_2$
occurs below, and it collapses:
\begin{equation}\label{eq:half}
 1-S+\frac{\delta_2}{2}=\frac12 .
\end{equation}

\section{Evaluation of $d$, $r$ and $\gamma$}\label{sec:evaluation}

The factorization \cref{eq:Winv} expresses $W$ through $K^{-1}$ once the
columns of $V^{-1}D^{-1}$ have been identified.  Set $b_j=D^{-1}V^{-1}e_j$ for
$0\le j\le n$.

\begin{lemma}\label[lemma]{lem:lagrange}
$Gb_n=v_1$, \quad $Gb_{n-1}=v_3-Sv_1$ \quad and \quad
$Gb_0=-\Delta^{2}v_{-1}$.
\end{lemma}

\begin{proof}
The Lagrange polynomial $\ell_i(X)=g(X)/\bigl(g'(x_i)(X-x_i)\bigr)$ takes the
value $1$ at $x_i$ and vanishes at the other nodes.  As $(Vc)_i$ evaluates at
$x_i$ the polynomial with coefficient vector $c$, the coefficient vector of
$\ell_i$ is $V^{-1}e_i$.  Both $V$ and hence $V^{-1}$ are symmetric, so
\[
 (b_j)_i=g'(x_i)(V^{-1})_{ij}=g'(x_i)(V^{-1})_{ji}
 =[X^{j}]\,\frac{g(X)}{X-x_i} .
\]
Now $g(X)/(X-x_i)=\prod_{k\ne i}(X-x_k)$ is monic of degree $n$.  Its leading
coefficient is $1$, so $b_n=\one$ and $Gb_n=v_1$.  Its coefficient of
$X^{n-1}$ is $-\sum_{k\ne i}x_k=x_i-S=a_i^{2}-S$, and multiplying by $a_i$
gives $a_i^{3}-Sa_i$, whence $Gb_{n-1}=v_3-Sv_1$.  Its constant term is
$(-1)^{n}\prod_{k\ne i}x_k=-\Delta^{2}/x_i$, because $n$ is odd and
$\prod_{k}x_k=\prod_{k}a_k^{2}=\Delta^{2}$; multiplying by $a_i$ gives
$-\Delta^{2}a_i^{-1}$, that is, $Gb_0=-\Delta^{2}v_{-1}$.
\end{proof}

Since $V^{-1}$ is symmetric and $D$ is diagonal, we have
$(V^{-1}D^{-1})^{\T}=D^{-1}V^{-1}$, so \cref{eq:Winv} reads
\begin{equation}\label{eq:Wentry}
 W_{jk}=\lambda^{-1}(Gb_j)^{\T}K^{-1}(Gb_k).
\end{equation}
Moreover $Ve_0=\one$, so $V^{-1}\one=e_0$ and $GD^{-1}V^{-1}\one=he_0$, where
\[
 h=g'(1)=\prod_{i=1}^{n}(1-\zeta^{2i}),
\]
and \cref{eq:Winv} therefore also gives
\begin{equation}\label{eq:gammaentry}
 \gamma_j=(W\one)_j=\frac h\lambda\,(Gb_j)^{\T}K^{-1}e_0 .
\end{equation}

Both \cref{eq:Wentry,eq:gammaentry} apply a linear functional to the vectors
of \cref{lem:lagrange}, and those vectors are precisely the ones produced by
expanding $k(z)$.

\begin{lemma}\label[lemma]{lem:extract}
Let $\phi$ be a linear functional on $\mathbb C^{n+1}$ and put
$f(z)=\phi(k(z))$.  Then $\phi(v_1)=-f(0)$, $\phi(v_{-1})=-f(\infty)$ and
$\phi(v_3-Sv_1)=-[z^{2}]f-(1-S)f(0)$.  In particular, if $f(0)\ne0$ then
\[
 \frac{\phi(v_3-Sv_1)}{\phi(v_1)}=1-S+\frac{[z^{2}]f}{f(0)} .
\]
\end{lemma}

\begin{proof}
Expanding each entry of $k(z)$ at the origin,
\[
 k(z)=-v_1+z(v_0-v_2)+z^{2}(v_1-v_3)+O(z^{3}),
\]
so $f(0)=-\phi(v_1)$ and $[z^{2}]f=\phi(v_1-v_3)$.  Eliminating $\phi(v_3)$
between these two relations gives
$\phi(v_3-Sv_1)=-f(0)-[z^{2}]f+Sf(0)$, which is the third formula, and
$\phi(v_{-1})=-f(\infty)$ because $k(\infty)=-v_{-1}$.
\end{proof}

We apply \cref{lem:extract} to the two functionals
\[
 \phi_1(y)=v_{-1}^{\T}K^{-1}y,\qquad \phi_2(y)=y^{\T}K^{-1}e_0,
\]
whose generating functions \cref{lem:kernel} has already produced, in
\cref{eq:e0,eq:infty}:
\begin{align*}
 f_1(z)&=\frac{1-\Delta^{-1}P(z)}{z}
   =-\rho_1-\rho_2z-\rho_3z^{2}+O(z^{3}),\\
 f_2(z)&=P(z)=\Delta\bigl(1+\rho_1z+\rho_2z^{2}+O(z^{3})\bigr).
\end{align*}
Combining \cref{lem:lagrange,lem:extract} with
\cref{eq:Wentry,eq:gammaentry}, and using $\rho_1=\delta_1=\tau$ from
\cref{eq:delta13}, the quantities we are after become
\begin{align}
 d&=\lambda^{-1}(Gb_0)^{\T}K^{-1}(Gb_n)
   =-\lambda^{-1}\Delta^{2}\phi_1(v_1)
   =-\frac{\Delta^{2}\tau}{\lambda},
 \label{eq:d}\\
 \frac rd&=\frac{\phi_1(v_3-Sv_1)}{\phi_1(v_1)}
   =1-S+\frac{\rho_3}{\rho_1},
 \label{eq:rd}\\
 \gamma_n&=\frac h\lambda\,\phi_2(v_1)=-\frac{h\Delta}{\lambda},
 \label{eq:gamman}\\
 \frac{\gamma_0}{\gamma_n}&=\frac{-\Delta^{2}\phi_2(v_{-1})}{\phi_2(v_1)}
   =\frac{-\Delta^{2}\cdot(-\Delta^{-1})}{-\Delta}=-1,\\
 \frac{\gamma_{n-1}}{\gamma_n}&=\frac{\phi_2(v_3-Sv_1)}{\phi_2(v_1)}
   =1-S+\rho_2 .
 \label{eq:gammanm1}
\end{align}
Thus $\gamma_0=-\gamma_n$, while \cref{eq:rho,eq:half,eq:gauss} turn
\cref{eq:gammanm1} into
\[
 \frac{\gamma_{n-1}}{\gamma_n}
 =1-S+\frac{\delta_2}{2}+\frac{\delta_1^{2}}{2}
 =\frac12+\frac{\tau^{2}}{2}
 =\frac{1-p}{2}=-n,
\]
that is, $\gamma_{n-1}=-n\gamma_n$.  The same three ingredients applied to
\cref{eq:rd} give
\[
 \frac rd
 =1-S+\frac{\delta_2}{2}+\frac{\delta_1^{2}}{6}+\frac{\delta_3}{3\delta_1}
 =\frac12-\frac p6+\frac{\chi(3)}{3}.
\]
Here quadratic reciprocity enters, in the form
$\chi(3)\bigl(\tfrac p3\bigr)=(-1)^{(p-1)/2}=(-1)^{n}=-1$.  If $p=6m+1$ then
$\bigl(\tfrac p3\bigr)=1$, hence $\chi(3)=-1$; if $p=6m+5$ then
$\bigl(\tfrac p3\bigr)=\bigl(\tfrac 23\bigr)=-1$, hence $\chi(3)=1$.  The two
cases give
\[
 \frac12-\frac{6m+1}{6}-\frac13=-m
 \qquad\text{and}\qquad
 \frac12-\frac{6m+5}{6}+\frac13=-m
\]
respectively, so that $r=-md$ either way.

\begin{remark}
The $3$ in \cref{eq:main} enters at exactly one place.  The cubic coefficient
$\rho_3$ contains the power sum $\delta_3$, which is the Gauss sum
$\tau(3)=\chi(3)\tau$, and the $3$ in the resulting term
$\delta_3/(3\delta_1)$ survives into $r/d$.  Quadratic reciprocity then
converts $\chi(3)$ into the residue of $p$ modulo $6$, leaving $m=\floor{p/6}$.
\end{remark}

Only the two squares remain.  Comparing \cref{eq:d,eq:gamman},
\[
 \frac{\gamma_n}{d}=\frac{-h\Delta/\lambda}{-\Delta^{2}\tau/\lambda}
 =\frac{h}{\Delta\tau},
\]
so it suffices to know $h^{2}$.  Evaluating
$\prod_{r=1}^{p-1}(X-\zeta^{r})=1+X+\cdots+X^{p-1}$ at $X=1$ gives
$p=\prod_{r=1}^{p-1}(1-\zeta^{r})$, and since the classes $\pm2i$ with
$1\le i\le n$ run exactly once through the nonzero residues modulo $p$, this
product factors as
\[
 p=\prod_{i=1}^{n}(1-\zeta^{2i})\prod_{i=1}^{n}(1-\zeta^{-2i})
 =(-1)^{n}\zeta^{-n(n+1)}h^{2}
 =-\zeta^{-n(n+1)}h^{2},
\]
the middle step using $1-\zeta^{-2i}=-\zeta^{-2i}(1-\zeta^{2i})$ and
$2(1+2+\cdots+n)=n(n+1)$.  Since
$\Delta^{2}=\prod_{i=0}^{n}a_i^{2}=\zeta^{n(n+1)}$, this says
$h^{2}=-p\Delta^{2}=\tau^{2}\Delta^{2}$ by \cref{eq:gauss}, whence
\[
 \left(\frac{\gamma_n}{d}\right)^{2}=\frac{h^{2}}{\Delta^{2}\tau^{2}}=1,
 \qquad\text{so that}\qquad \gamma_n^{2}=d^{2}.
\]
There remains $d^{2}$ itself.  From $2n(n+1)=n(p+1)\equiv n\pmod p$ we get
$\Delta^{4}=\zeta^{n}$, while $\tau(2)=\chi(2)\tau$ and
$-(p+1)/2\equiv n\pmod p$ give
$\lambda^{2}=\tau(2)^{2}\zeta^{-(p+1)/2}=-p\zeta^{n}$.  Hence, by
\cref{eq:d,eq:gauss},
\[
 d^{2}=\frac{\Delta^{4}\tau^{2}}{\lambda^{2}}
 =\frac{\zeta^{n}\cdot(-p)}{-p\zeta^{n}}=1 .
\]
This completes the proof of \cref{prop:inverse}, and with it, through
\cref{prop:reduction}, the proof of \cref{thm:main}. \qed

\backmatter

\section*{Declarations}

\noindent\textbf{Use of generative artificial intelligence.}
The proof given here, together with a first draft of this manuscript, was
produced by OpenAI's ChatGPT.  The authors verified every mathematical step
independently, rewrote the exposition, and take full responsibility for the
contents of the paper.

\bibliography{references}

\end{document}